\documentclass[12pt]{amsart}
\usepackage{amsfonts, amsthm, amsmath,dsfont}
\usepackage{rotating}
\usepackage{tikz}
\usepackage{graphics}
\usepackage{amssymb}
\usepackage{amscd}
\usepackage[latin2]{inputenc}
\usepackage{t1enc}
\usepackage[mathscr]{eucal}
\usepackage{indentfirst}
\usepackage{graphicx}
\usepackage{graphics}
\usepackage{pict2e}
\usepackage{mathrsfs}
\usepackage{enumerate}
\usepackage{tikz}
\usetikzlibrary{tikzmark,arrows.meta}
\usepackage[pagebackref]{hyperref}
\hypersetup{colorlinks=true}
\usepackage{cite}
\usepackage{color}
\usepackage{epic}
\usepackage{hyperref}
\usepackage{framed}
\usepackage{mathabx}
\usepackage{booktabs}
\usepackage{makecell}
\usepackage{comment}
\newcolumntype{V}{!{\vrule width 2pt}}
\numberwithin{equation}{section}
\def\blue{\textcolor{blue}}
\def\red{\textcolor{red}}

\theoremstyle{plain}
\newtheorem{theorem}{Theorem}[section]

\newtheorem{remark}[theorem]{Remark}
\newtheorem{lemma}[theorem]{Lemma}
\newtheorem{definition}[theorem]{Definition}

\def\A{\mathcal{A}}
\def\U{\mathcal{U}}
\def\IP{\mathcal{IP}}
\def\UP{\mathcal{UP}}
\def\CP{\mathcal{CP}}

\def\inv{\mathsf{inv}}
\def\des{\mathsf{des}}

\def\h{\mathsf{h}}
\def\ind{\mathsf{ind}}
\def\wt{\mathrm{wt}}
\def\sgn{\mathsf{sgn}}
\def\LB{\mathcal{LB}}
\def\IB{\mathcal{IB}}

\begin{document}
\title[An involution for a Catalan-tangent number identity]{An involution for
a Catalan-tangent number identity}
\author[D. Kim]{Dongsu Kim}
\address[Dongsu Kim]{Department of Mathematical Sciences, Gwangju Institute of Science
and Technology, Gwangju 61005, Republic of Korea}
\email{dongsukim@gist.ac.kr, dongsu.kim@kaist.ac.kr}
\author[Z. Lin]{Zhicong Lin}
\address[Zhicong Lin]{Research Center for Mathematics and Interdisciplinary Sciences,
Shandong University \& Frontiers Science Center for Nonlinear Expectations,
Ministry of Education, Qingdao 266237, P.R. China}
\email{linz@sdu.edu.cn}
\date{\today}
\begin{abstract} We provide an involution proof of a Catalan-tangent number identity arising from
the study of peak algebra that was found by Aliniaeifard and Li. In the course, we find a new
combinatorial identity for the tangent numbers $T_{2n+1}$:
$$
\sum_{k=0}^{n}(-1)^{k}{2n+1\choose 2k}2^{2n-2k}T_{2k+1}=(-1)^nT_{2n+1}.
$$
Moreover, we derive two different $q$-analogs of the above identity from the combinatorial perspective.
\end{abstract}
\maketitle
\section{Introduction}
For a positive integer $n$, let $[n]:=\{1,2,\ldots,n\}$. A {\em set composition} $\phi$ of $[n]$,
denoted by $\phi\vDash[n]$, is a list of mutually disjoint nonempty subsets
$\phi_1/\phi_2/\ldots/\phi_{\ell}$ of $[n]$ whose union is $[n]$.
Each $\phi_i$ is called a {\em block} of $\phi$ and the number of blocks of $\phi$ is denoted
by $\ell(\phi)$.
A set composition $\phi$ is said to be {\em odd} if each block of $\phi$ consists of odd number
of elements. For instance, $\phi=7/148/9/6/235$ is an odd set composition of $[9]$ with $\ell(\phi)=5$. 
Let
$$
O(n,k):=|\{\phi\vDash[n]: \phi\text{ odd}, \ell(\phi)=k\}|.
$$
Note that $O(n,k)=0$ if $n$ and $k$ have different parity. By the compositional formula for
the exponential generating functions (see~\cite[Proposition~5.1.3]{St}), we have 
\begin{align*}
&\quad\sum_{n,k} O(n,k)t^k\frac{x^n}{n!}=\frac{t\sinh(x)}{1-t\sinh(x)}\\
&=tx+2t^2\frac{x^2}{2!}+(t+6t^3)\frac{x^3}{3!}+(8t^2+24t^4)\frac{x^3}{3!}+
(t+60t^3+120t^5)\frac{x^5}{5!}+\cdots.
\end{align*}

The {\em Catalan numbers} 
$$
\left(C_n=\frac{1}{n+1}{2n\choose n}\right)_{n\geq0}=(1,1,2,5,14,42,132,\ldots)
$$
and the {\em tangent numbers} $\{T_{2n+1}\}_{n\geq0}$ defined by the Taylor expansion
$$
\tan(x)=\sum_{n\geq0}T_{2n+1}\frac{x^{2n+1}}{(2n+1)!}=
x+2\frac{x^3}{3!}+16\frac{x^5}{5!}+272\frac{x^7}{7!}+7936\frac{x^9}{9!}+\cdots
$$
are well-known combinatorial sequences in enumerative combinatorics.

The following identity, which connects Catalan numbers with tangent numbers via $O(n,k)$,
was found by Aliniaeifard and Li~\cite{AL} in their study of non-commutative peak algebra
and later proved by Zhao, Lin and Zang~\cite{ZLZ} using generating functions.

\begin{theorem}\label{thm:Cat-Tan}
For $n\geq0$, we have 
\begin{equation}\label{eq:Cat-Tan}
\sum_{k=0}^n (-1)^kO(2n+1,2k+1)2^{2n-2k}C_k=(-1)^nT_{2n+1}. 
\end{equation}
\end{theorem}

For instance, when $n=2$ identity~\eqref{eq:Cat-Tan} reads 
$$
1\cdot2^4C_0-60\cdot2^2C_1+120\cdot2^0C_2=16.
$$
Although Catalan numbers and tangent numbers have been extensively studied in combinatorics
(see~\cite[Exercise~6.19]{St} and~\cite{St10}), there seems to be no combinatorial identity relating
them other than~\eqref{eq:Cat-Tan}. As observed by Aliniaeifard and Li~\cite{AL},
combining~\cite[Corollary~11.5]{AL} and
Theorem~\ref{thm:Cat-Tan} results in the following intriguing identity for tangent numbers 
\begin{equation}\label{eq:Genocchi}
\sum_{k=0}^{n-1}(-1)^{k}{2n\choose 2k+1}2^{2n-2k}T_{2k+1}=2^{2n+1}.
\end{equation}
A generating function proof of Theorem~\ref{thm:Cat-Tan} and a combinatorial proof of~\eqref{eq:Genocchi},
together with its surprising arithmetic applications, were given in~\cite{ZLZ}.
However, a combinatorial proof of Theorem~\ref{thm:Cat-Tan} remains elusive. 

The main objective of this paper is to provide a combinatorial involution proof of
Theorem~\ref{thm:Cat-Tan}.
In the course, we find a new combinatorial identity similar to~\eqref{eq:Genocchi}:
\begin{equation}\label{eq:Tan2}
\sum_{k=0}^{n}(-1)^{k}{2n+1\choose 2k}2^{2n-2k}T_{2k+1}=(-1)^nT_{2n+1}.
\end{equation}
This identity is also reminiscent of the one proved by Andrews and Gessel~\cite{AG}: 
\begin{equation}\label{eq:TanAG}
T_{2n+1}+\sum_{k=1}^{n}(-1)^{k}{2n+1\choose 2k}2^{2k-1}T_{2n-2k+1}=(-1)^n2^{2n}.
\end{equation}

The rest of this paper is organized as follows. In Section~\ref{sec:2}, we introduce the combinatorial
structure of complete binary trees labeled by odd unimodal permutations and then construct
a sign-reversing involution on them in Section~\ref{sec:3} to prove Theorem~\ref{thm:Cat-Tan}
combinatorially. In Section~\ref{sec:4}, we prove two different $q$-analogs of~\eqref{eq:Tan2}
using our involution on labeled binary trees and some other involutions on permutation pairs. 
\section{Binary trees labeled by odd unimodal permutations}
\label{sec:2}
A {\em complete binary tree} is a rooted tree in which every internal node has a left child and
a right child. It is well known (see~\cite[Exercise~6.19]{St}) that $C_n$ enumerates the number
of complete binary trees with $2n+1$ nodes. A {\em complete increasing binary tree} on $[2n+1]$ is
a labeled complete binary tree such that the labels along every path from the root to a leaf
are increasing. Complete increasing binary trees on $[2n+1]$ are known~\cite{St10} to be counted
by $T_{2n+1}$. In order to interpret the left-hand side of~\eqref{eq:Cat-Tan}, we need to
generalize complete increasing binary trees to some complete binary trees labeled by odd
unimodal permutations which are defined below. 

A word $w=w_1w_2\cdots w_k$ with distinct letters is called a {\em unimodal permutation} if 
$$
w_1<w_2<\cdots<w_{m}>w_{m+1}>\cdots >w_k
$$ 
for some $1\leq m\leq k$. If $k$ is odd, then $w$ is said to be an {\em odd unimodal permutation}. 
We introduce the following generalization of complete increasing binary trees. 
\begin{definition}[Labeled binary trees] 
A \blue{{\bf\em labeled (complete) binary tree on $[2n+1]$}} is a complete binary tree whose nodes are
labeled by odd unimodal permutations, and the letters of all labels form a set partition of $[2n+1]$.
See Fig.~\ref{ex:lbt} for three labeled binary trees on $[9]$. 
\end{definition}

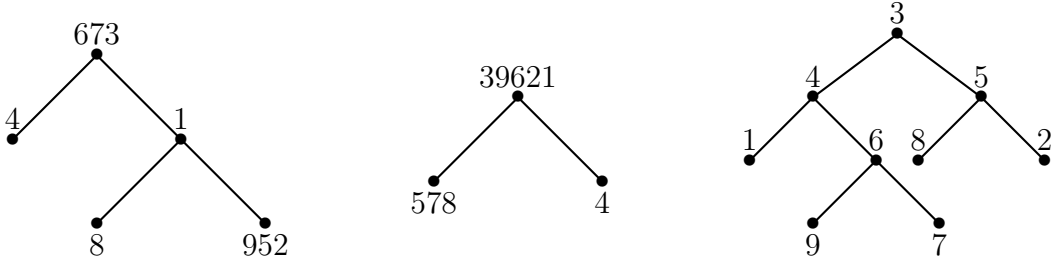
\begin{figure}
\centering
\begin{tikzpicture}[scale=0.28]
\draw[-,thick](10,10) to (6,6);
\draw[-,thick](10,10) to (14,6);
\draw[-,thick](14,6) to (10,2);
\draw[-,thick](14,6) to (18,2);
\node at (10,10){$\bullet$};
\node at (10,11){$673$};
\node at (6,6){$\bullet$};
\node at (6,7){$4$};
\node at (14,6){$\bullet$};\node at (14,7){$1$};
\node at (10,2){$\bullet$};\node at (10,1){$8$};
\node at (18,2){$\bullet$};\node at (18,1){$952$};
\draw[-,thick](30,8) to (26,4);\draw[-,thick](30,8) to (34,4);
\node at (30,8){$\bullet$};\node at (30,9){$39621$};
\node at (26,4){$\bullet$};\node at (26,3){$578$};
\node at (34,4){$\bullet$};\node at (34,3){$4$};
\draw[-,thick](48,11) to (44,8);\node at (48,11){$\bullet$};\node at (48,12){$3$};
\node at (44,8){$\bullet$};\node at (44,9){$4$};
\draw[-,thick](44,8) to (41,5);
\node at (41,5){$\bullet$};\node at (41,6){$1$};
\draw[-,thick](44,8) to (47,5);
\node at (47,5){$\bullet$};\node at (47,6){$6$};
\draw[-,thick](47,5) to (44,2);
\node at (44,2){$\bullet$};\node at (44,1){$9$};
\draw[-,thick](47,5) to (50,2);
\node at (50,2){$\bullet$};\node at (50,1){$7$};
\draw[-,thick](48,11) to (52,8);
\node at (52,8){$\bullet$};\node at (52,9){$5$};
\draw[-,thick](52,8) to (49,5);
\node at (49,5){$\bullet$};\node at (49,6){$8$};
\draw[-,thick](52,8) to (55,5);
\node at (55,5){$\bullet$};\node at (55,6){$2$};
\end{tikzpicture}
\caption{Three labeled binary trees on $[9]$.\label{ex:lbt}}
\end{figure}

Let $\LB_{2n+1}$ be the set of all labeled binary trees on $[2n+1]$ and let $\IB_{2n+1}$ be the set of
all complete increasing binary trees on $[2n+1]$. It is clear that $\IB_{2n+1}\subseteq\LB_{2n+1}$.
For a labeled binary tree $T\in\LB_{2n+1}$, let $\h(T)$ be half the number of edges of $T$.
The following lemma shows that the left-hand side of~\eqref{eq:Cat-Tan} is
a signed counting of labeled binary trees on $[2n+1]$. 
\begin{lemma}\label{lem:sign}
For $n\geq0$, we have 
\begin{equation}\label{eq:sign}
\sum_{T\in\LB_{2n+1}}(-1)^{\h(T)}=\sum_{k=0}^n (-1)^kO(2n+1,2k+1)2^{2n-2k}C_k.
\end{equation}
\end{lemma}
\begin{proof}
Each labeled binary tree in $\LB_{2n+1}$ with $2k$ edges can be constructed by 
\begin{itemize}
\item choosing a complete binary tree $T$ with $2k+1$ nodes,
\item choosing a set composition $\phi$ of $[2n+1]$ with $2k+1$ blocks,
\item forming an odd unimodal permutation using letters in $\phi_i$ for the label of the $i$-th node
(under the in-order traversal) of $T$ for each $1\leq i\leq 2k+1$.
\end{itemize}
There are $C_k$ choices for $T$, $O(2n+1,2k+1)$ choices for $\phi$, and $2^{2n+1-(2k+1)}$ ways to form
the labels of $T$ from $\phi$, whence~\eqref{eq:sign} follows.
\end{proof}

In view of Lemma~\ref{lem:sign}, it remains to construct a sign-reversing involution on labeled
binary trees fixing complete increasing binary trees to finish our combinatorial proof of
Theorem~\ref{thm:Cat-Tan}. 

\section{A sign-reversing involution on labeled binary trees for Theorem~\ref{thm:Cat-Tan}}
\label{sec:3}
A crucial result of this paper is the following lemma. 
\begin{lemma}\label{lem:invo}
There exists an involution $\kappa$ on $\LB_{2n+1}$ such that 
\begin{itemize}
\item $\kappa(T)=T$ if and only if $T\in\IB_{2n+1}$;
\item and $(-1)^{\h(T)}+(-1)^{\h(\kappa(T))}=0$ for any $T\in\LB_{2n+1}\setminus\IB_{2n+1}$.
\end{itemize}
\end{lemma}

Combining Lemma~\ref{lem:sign} with Lemma~\ref{lem:invo}, we can obtain a combinatorial proof of
Theorem~\ref{thm:Cat-Tan}. The rest of this section is devoted to the construction of $\kappa$. 

Let $T\in\LB_{2n+1}$ be a labeled binary tree. For a node $v$ of $T$, let $T_v$ be the subtree
of $T$ rooted at $v$. We say that $v$ is {\em active}, if $T_v$ is not an increasing subtree and
either $v$ is a leaf\footnote{This happens when the label of $v$ is an odd unimodal
permutation of size greater than one.} or both branches of $T_v$ are increasing subtrees.
For the three labeled binary trees in Fig.~\ref{ex:lbt}, the only active node in the first tree is labeled
$952$, the only active node in the second tree is labeled $578$, while the third tree contains two active
nodes which are labeled $4$ and $5$. Given a binary tree, its {\em right chain} (resp., {\em left chain})
is any maximal path composed of only right (resp., left) edges. 
\smallskip

{\bf The construction of $\kappa$.} Given $T\in\LB_{2n+1}$, if $T\in\IB_{2n+1}$, then define $\kappa(T)=T$.
Otherwise, find the first active node $v$ under the in-order traversal. Suppose that the label of $v$ is
the odd unimodal permutation $\pi=\pi_1\pi_2\cdots\pi_{2k+1}$. We define $\kappa(T)$ to be the tree
obtained from $T$ by replacing $T_v$ by $T_v^*$, where $T_v^*$ is constructed according to
the following cases. 
\begin{itemize}
\item[(I)] If $v$ is a leaf, then $k\geq1$. Let $T_v^*$ be the tree with root $\pi_2\cdots\pi_{2k}$
whose left child is $\pi_1$ and right child is $\pi_{2k+1}$. See Fig.~\ref{case:I-II} for an example. 
\item[(II)] If $v$ has left (resp.,~right) child labeled by $a$ (resp.,~$b$) and both $a,b$ are leaves
such that $a\pi b$ is unimodal, then let $T_v^*$ be the tree with a unique node labeled by $a\pi b$.
Note that this case is the inverse of case (I); see Fig.~\ref{case:I-II}. 
\begin{figure}
\centering
\begin{tikzpicture}[scale=0.28]
\draw[-,thick](30,8) to (26,4);\draw[-,thick](30,8) to (34,4);
\node at (30,8){$\bullet$};\node at (30,9.3){$57642$};
\node at (26,4){$\bullet$};\node at (26,2.7){$3$};
\node at (34,4){$\bullet$};\node at (34,2.7){$1$};
\node at (39,7){$\longrightarrow$};\node at (39,8.1){\small{(II)}};
\node at (39,6){$\longleftarrow$};\node at (39,5){\small{(I)}};
\node at (46.5,6.5){$\bullet$};\node at (46.5,7.8){$3576421$};
\end{tikzpicture}
\caption{Examples for cases (I) and (II).\label{case:I-II}}
\end{figure}
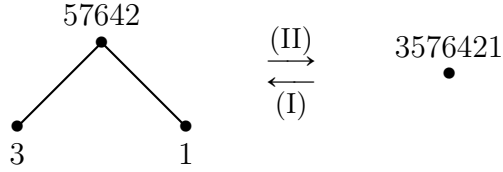
\item[(III)] Cases other than (I) and (II). Let $a_0$ be the label of the left child of $v$, and
$b_0$ the label of the right child of $v$. Find the right (resp.,~left) chain starting from $a_0$
(resp.,~$b_0$), whose labels are $a_0,a_1,\ldots,a_{s}$ (resp.,~$b_0,b_1,\ldots,b_{t}$)
from top to bottom. Introduce the {\em left indicator} $\ind_l$ as:
$$
\ind_l=
\begin{cases}
+\infty,&\text{if $k=0$ and $s=0$;}\\ 
a_0,&\text{if $k=0$ and $s>0$;}\\
\min\{a_0,\pi_1\},&\text{if $k>0$ and ($a_s>\pi_1$ and $\pi_1<\pi_2$);}\\
+\infty,&\text{if $k>0$ and $s=0$ and ($a_0<\pi_1$ or $\pi_1>\pi_2$);}\\
a_0,&\text{if $k>0$ and $s>0$ and ($a_s<\pi_1$ or $\pi_1>\pi_2$).}
\end{cases}
$$
Symmetrically, we define the {\em right indicator} $\ind_r$ as:
$$
\ind_r=
\begin{cases}
+\infty,&\text{if $k=0$ and $t=0$;}\\ 
b_0,&\text{if $k=0$ and $t>0$;}\\
\min\{b_0,\pi_{2k+1}\},&\text{if $k>0$ and ($b_t>\pi_{2k+1}$ and $\pi_{2k+1}<\pi_{2k}$);}\\
+\infty,&\text{if $k>0$ and $t=0$ and ($b_0<\pi_{2k+1}$ or $\pi_{2k+1}>\pi_{2k}$);}\\
b_0,&\text{if $k>0$ and $t>0$ and ($b_t<\pi_{2k+1}$ or $\pi_{2k+1}>\pi_{2k}$).}
\end{cases}
$$
Note that $\ind_l=\ind_r=+\infty$ could not happen.\footnote{Refer to Case (II).} We need to distinguish the following two cases.
\begin{itemize}
\item[(a)] $\ind_l<\ind_r$. We further distinguish two subcases.
\begin{itemize}
\item[(a-1)] If $k>0$ and ($a_s>\pi_1$ and $\pi_1<\pi_2$), then find the smallest index $i$,
$0\leq i\leq s$, such that $a_i>\pi_1$. Suppose that $u$ is the node labeled by $a_{i}$.
In this case, let $T_v^*$ be obtained from $T_v$ by first changing $v$'s label to
$\pi_3\pi_4\cdots\pi_{2k+1}$
and then replacing the subtree $T_{u}$ with the new subtree $\widetilde{T_u}$,
where $\widetilde{T_u}$ is the rooted binary tree with root labeled by $\pi_1$ and the left branch
of the root is $T_u$, while the right branch of the root is the leaf labeled by $\pi_2$.
See Fig.~\ref{case:a12} for an example. 
\begin{figure}
\centering
\begin{tikzpicture}[scale=0.28]
\draw[-,thick](6,9) to (3,6);
\draw[-,thick](6,9) to (9,6);
\node at (6,9){$\bullet$};\node at (6,10){$\red{35}2$};
\node at (9,6){$\bullet$};\node at (9,7){$8$};
\node at (3,6){$\bullet$};\node at (3,7){$1$};
\draw[-,thick](3,6) to (0,3);\draw[-,thick](3,6) to (6,3);
\node at (0,3){$\bullet$};\node at (0,4){$9$};
\blue{\node at (6,3){$\bullet$};\node at (6,4){$4$};
\draw[-,thick](6,3) to (3,0);\draw[-,thick](6,3) to (9,0);
\node at (3,0){$\bullet$};\node at (3,-1){$7$};
\node at (9,0){$\bullet$};\node at (9,-1){$6$};}
\node at (14.5,6){$\longrightarrow$};\node at (14.5,7.2){(a-1)};
\node at (14.5,5){$\longleftarrow$};\node at (14.5,4){(a-2)};
\draw[-,thick](26,10) to (23,7);
\draw[-,thick](26,10) to (29,7);
\node at (26,10){$\bullet$};\node at (26,11){$2$};
\node at (29,7){$\bullet$};\node at (29,8){$8$};
\node at (23,7){$\bullet$};\node at (23,8){$1$};
\draw[-,thick](23,7) to (20,4);\draw[-,thick](23,7) to (26,4);
\node at (20,4){$\bullet$};\node at (20,5){$9$};
\red{\node at (26,4){$\bullet$};\node at (26,5){$3$};
\draw[-,thick](26,4) to (23,1);\draw[-,thick](26,4) to (29,1);
\node at (29,1){$\bullet$};\node at (29,2){$5$};}
\blue{\node at (23,1){$\bullet$};\node at (23,2){$4$};
\draw[-,thick](23,1) to (20,-2);\draw[-,thick](23,1) to (26,-2);
\node at (20,-2){$\bullet$};\node at (20,-3){$7$};
\node at (26,-2){$\bullet$};\node at (26,-3){$6$};}
\end{tikzpicture}
\caption{Examples for subcases (a-1) and (a-2).\label{case:a12}}
\end{figure}
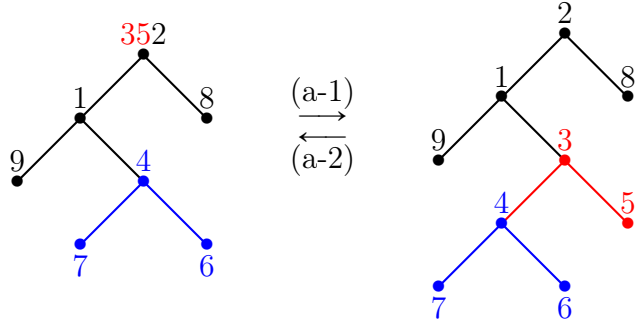
\item[(a-2)] Otherwise, suppose that $u$ is the node labeled by $a_{s-1}$. In this case,
let $T_v^*$ be obtained from $T_v$ by first changing $v$'s label to $a_{s-1}a_s\pi$ and then
replacing the subtree $T_{u}$ with the subtree $T_{u'}$, where $u'$ is the left child of $u$.
Note that this case is the inverse of case (a-1); see Fig.~\ref{case:a12}. 
\end{itemize}
\item[(b)] $\ind_l>\ind_r$. This case is symmetric to case (a) and we further consider two subcases.
\begin{itemize}
\item[(b-1)] If $k>0$ and ($b_t>\pi_{2k+1}$ and $\pi_{2k+1}<\pi_{2k}$), then find the smallest index
$j$, $0\leq j\leq t$, such that $b_j>\pi_{2k+1}$. Suppose that $u$ is the node labeled by $b_{j}$. 
In this case, let $T_v^*$ be obtained from $T_v$ by first changing $v$'s label to
$\pi_1\pi_2\cdots\pi_{2k-1}$ and then replacing the subtree $T_{u}$ with the new subtree
$\widetilde{T_u}$, where $\widetilde{T_u}$ is the rooted binary tree with root labeled
by $\pi_{2k+1}$ and the right branch of the root is $T_u$, while the left branch of the root
is the leaf labeled by $\pi_{2k}$. 
\item[(b-2)] Otherwise, suppose that $u$ is the node labeled by $b_{t-1}$. In this case,
let $T_v^*$ be obtained from $T_v$ by first changing $v$'s label to $\pi b_{t}b_{t-1}$ and
then replacing the subtree $T_{u}$ with the subtree $T_{u'}$, where $u'$ is the right child
of $u$. Note that this case is the inverse of case (b-1). 
\end{itemize}
\end{itemize}
\end{itemize}
\vskip0.1in
Now we are ready for the proof of Lemma~\ref{lem:invo}.

\begin{proof}[{\bf Proof of Lemma~\ref{lem:invo}}] It is clear from the construction that
$\kappa$ is sign-reversing and fixing all complete increasing binary trees.
Given $T\in\LB_{2n+1}\setminus\IB_{2n+1}$, it is plain to see that $\kappa^2(T)=T$ when
$v$ is in cases (I) or (II). For $v$ in case (III), to see that $\kappa$ is an involution,
all we need to show is that $\ind_l(T_v)<\ind_r(T_v)$ implies $\ind_l(T_v^*)<\ind_r(T_v^*)$
due to the symmetry of labeled binary trees. This is obviously true when the label of $v$
is an odd unimodal permutation of size at least five.
Suppose that the label of $v$ is of size $1$ or $3$ and $m$ is the smallest letter in all
the labels of $T_v$. Then $m=\min\{\pi_1,a_0,\pi_{2k+1},b_0\}$.
We consider the following two cases.
\begin{itemize}
\item $v$ is in case (a-1):
\item[] The label of $v$ in $T_v$ is $\pi=\pi_1\pi_2\pi_3$.
If $m=\min\{\pi_1,a_0\}$ then $\ind_l(T_v^*)=m<\ind_r(T_v^*)$; otherwise, i.e.,
if $m=\min\{\pi_{3},b_0\}$ then either $m=b_0$ or $m=\pi_{3}$, but the latter implies
$\ind_r(T_v)=m$, which is a contradiction,
Hence we have $m=b_0$, and so $\ind_r(T_v)=+\infty$.
This implies that the node labeled by $b_0$ is a leaf in $T_v$,
which forces $\ind_r(T_v^*)=+\infty$. This proves that $\ind_l(T_v^*)<\ind_r(T_v^*)$. 
\item $v$ is in case (a-2):
\item[] If the label of $v$ in $T_v$ is $\pi=\pi_1\pi_2\pi_3$, then we have 
$\ind_l(T_v^*)=\ind_l(T_v)<\ind_r(T_v)=\ind_r(T_v^*)$. If the label of $v$ in $T_v$ is $\pi=\pi_1$,
then $m=\min\{\pi_1,a_0\}=a_0$ since $m=\pi_1$ implies that $T_v$ is an increasing binary tree,
which is a contradiction. Now $a_0=m$ implies that $\ind_l(T_v^*)=m<\ind_r(T_v^*)$. 
\end{itemize}
We have proved that $\ind_l(T_v)<\ind_r(T_v)$ implies $\ind_l(T_v^*)<\ind_r(T_v^*)$ in all cases.
Thus, $\kappa$ is an involution, which completes the proof of the lemma.
\end{proof}

\section{Two $q$-analogs of~\eqref{eq:Tan2}}
\label{sec:4}
In this section, we derive two $q$-analogs of~\eqref{eq:Tan2} from the
combinatorial perspective.

Before proceeding, we introduce some necessary definitions and notations on $q$-series.
Given a nonnegative integer 
$n$, the {\em$q$-shifted factorial} is defined by
$(t;q)_n :=\prod_{i=0}^{n-1}(1-tq^i)$. In particular, $$(-q;q)_n=(1+q)(1+q^2)\cdots(1+q^n)
$$ is a natural $q$-analog of $2^n$. Recall that the $q$-binomial coefficients $\left[{n\atop k}\right]$ are defined by
$$
{n\brack k} :=\frac{(q;q)_n }{(q;q)_{n-k}(q;q)_k}\qquad \textrm{for}\quad 0\leq k\leq n.
$$
 The {\em$q$-sine} $\sin_q(x)$ and {\em$q$-cosine} $\cos_q(x)$ are defined by
$$
\sin_q(x):=\sum_{n\geq0}(-1)^n\frac{x^{2n+1}}{(q;q)_{2n+1}} \quad\text{and}\quad\cos_q(x):=\sum_{n\geq0}(-1)^n\frac{x^{2n}}{(q;q)_{2n}}.
$$
The classical {\em$q$-tangent numbers} (see~\cite{AG,Fo,FH10}) $T_{2n+1}(q)$ then occur in the expansion of $\tan_q(x)$:
$$
\tan_q(x):=\frac{\sin_q(x)}{\cos_q(x)}=\sum_{n\geq0}T_{2n+1}(q)\frac{x^{2n+1}}{(q;q)_{2n+1}}.
$$
For convenience sake, we list the first few values of $T_{2n+1}(q)$: 
\begin{align*}
 T_1(q) &= 1 ,\\
 T_3(q) & = (1 + q)q ,\\
 T_5(q) & = (1 + q)^2 (1 + q^2)^2 q^2 ,\\
 T_7(q) &= (1 + q)^2 (1 + q^2) (1 + q^3)q^3 (1 + q + 3q^2 + 2q^3 + 3q^4 + 2q^5 + 3q^6 + q^7 + q^8).
 \end{align*}

A permutation $\pi=\pi_1\pi_2\cdots\pi_n$ of $[n]$ is {\em down-up} (or {\em alternating}) if
$$
\pi_1>\pi_2<\pi_3>\pi_4<\cdots.
$$
If the above inequalities are reversed, then $\pi$ is said to be {\em up-down}. 
Let $\A_n$ be the set of all down-up permutations of $[n]$. 
A result attributed to Andr\'e~\cite{An} asserts that
\begin{equation}\label{eul:egf}
1+\sum_{n\geq1}|\A_n|\frac{x^n}{n!}=\sec(x)+\tan(x). 
\end{equation}
Thus, $T_{2n+1}=|\A_{2n+1}|$. 
It is a classical result (see~\cite{St10}) that $T_{2n+1}(q)$ has the combinatorial interpretation
\begin{equation}\label{int:qtan}
T_{2n+1}(q)=\sum_{\pi\in\A_{2n+1}}q^{\inv(\pi)},
\end{equation}
where $\inv(\pi):=|\{(i,j)\in[n]^2: i<j,\pi_i>\pi_j\}|$ is the {\em inversion number} of a permutation $\pi$.
The following is our first $q$-analog of~\eqref{eq:Tan2}.
\begin{theorem}\label{thm:q-analog1}
For $n\geq0$, we have 
\begin{equation}\label{q1:Tan2}
\sum_{k=0}^{n}(-1)^{k}{2n+1\brack 2k}(-q;q)_{2n-2k}\widetilde T_{2k+1}(q)=(-1)^nT_{2n+1}(q),
\end{equation}
where $\widetilde T_{1}(q)=1$ and
\begin{equation}
\widetilde T_{2k+1}(q)=\sum_{i=0}^{k-1} {2k\brack 2i+1} T_{2i + 1}(q) T_{2k - 2i - 1}(q) \quad (k \geq 1).
\end{equation}
\end{theorem}

\begin{remark}
Foata~\cite{Fo} proved the following quadratic recursion for tangent numbers:
$$
T_{2n + 1} = \sum_{k=0}^{n - 1} {2n\choose 2k+1} T_{2k + 1} T_{2n - 2k - 1} \quad (n \geq 1),
$$
from which we have $\widetilde T_{2n+1}(1)=T_{2n+1}$. Thus, identity~\eqref{q1:Tan2} is indeed a $q$-analog of~\eqref{eq:Tan2}. 
\end{remark}

We need some preparations before we can prove Theorem~\ref{thm:q-analog1}. 
Let $\U_n$ be the set of all unimodal permutations of $[n]$. For instance,
$\U_3=\{123,132,231,321\}$. The following lemma, which enumerates unimodal permutations
by inversion numbers, was proved in~\cite{ZLZ}.

\begin{lemma}\label{lem:unimodal}
For $n\geq1$, we have
\begin{equation}
\sum_{\pi\in\U_n}q^{\inv(\pi)}=(-q;q)_{n-1}.
\end{equation}
\end{lemma}

For a labeled tree $T\in\LB_{2n+1}$ with $2k+1$ nodes, consider the word concatenation
$w(T)=\alpha_1\alpha_2\cdots\alpha_{2k+1}$, where $\alpha_i$ is the labeling of the unimodal
permutation of the $i$-th node (under the in-order traversal) of $T$. For instance,
if $T$ is the first tree in Fig.~\ref{ex:lbt}, then $w(T)=467381952$.
It is clear that $w(T)$ is a permutation of $[2n+1]$. The {\em inversion number of $T$},
denoted by $\inv(T)$, is thus defined as 
$$
\inv(T):=\inv(w(T)).
$$
The following interpretation of signed $q$-tangent numbers is a consequence of 
Lemma~\ref{lem:invo}.
\begin{lemma}\label{lem:newint}
For $n\geq0$, we have
\begin{equation}\label{new:tan}
\sum_{T\in\LB_{2n+1}}(-1)^{\h(T)}q^{\inv(T)}=(-1)^nT_{2n+1}(q).
\end{equation}
\end{lemma}
\begin{proof}It is plain to check that our involution $\kappa$ has the following feature:
$$
\inv(T)=\inv(\kappa(T)) \quad\text{for any $T\in\LB_{2n+1}$},
$$
that is, it preserves the inversion numbers of labeled trees. The result then follows from Lemma~\ref{lem:invo}, interpretation~\eqref{int:qtan} and the known fact that $T\mapsto w(T)$ establishes a one-to-one correspondence between $\IB_{2n+1}$ and $\A_{2n+1}$ preserving inversion numbers. 
\end{proof}
 We are ready for the proof of Theorem~\ref{thm:q-analog1}.
\begin{proof}[{\bf Proof of Theorem~\ref{thm:q-analog1}}]
Recall the well-known combinatorial interpretation for the $q$-binomial coefficients:
\begin{equation}\label{eq:qmul}
{n\brack k}=\sum_{({\mathcal A}, {\mathcal B})}q^{\inv({\mathcal A}, {\mathcal B})},
\end{equation}
summed over all set compositions $({\mathcal A}, {\mathcal B})$ of $[n]$ such that $|{\mathcal A}|=k$, and
$$\inv({\mathcal A}, {\mathcal B}):=|\{(a,b)\in{\mathcal A}\times{\mathcal B}: a>b\}|.$$
A labeled binary tree $T\in\LB_{2n+1}$ can be decomposed as $(T_g,r,T_d)$, where $r$ is the root and $T_g$ and $T_d$ (possibly empty) are left branch and right branch of $r$, respectively. It follows from this decomposition, Lemma~\ref{lem:unimodal}, the interpretation~\eqref{new:tan} of $q$-tangent numbers and the interpretation~\eqref{eq:qmul} of $q$-binomial coefficients that
$$
(-1)^nT_{2n+1}(q)=(-q;q)_{2n}+\sum_{k=1}^{n}(-1)^{k}{2n+1\brack 2k}(-q;q)_{2n-2k}\sum_{i=0}^{k-1}{2k\brack 2i+1}T_{2i+1}(q)T_{2k-2i-1}(q),
$$
where $(-q;q)_{2n}$ counts single-node labeled binary trees on $[2n+1]$ by inversion numbers. 
This completes the proof of the theorem. 
\end{proof}

\subsection{Permutation pairs and another $q$-analog of~\eqref{eq:Tan2}}

In view of~\eqref{eul:egf}, the number $|\A_{2n}|:=S_{2n}$ is known as a {\em secant number}. 
Consider the {\em$q$-secant number }
$$
S_{2n}(q):=\sum_{\pi\in\A'_{2n}}q^{\inv(\pi)},
$$
where $\A'_n$ denotes the set of all up-down permutations of $[n]$. For convenience, set $S_0(q)=1$. It was known (see~\cite{FH10}) that 
$$
\sum_{n\geq0}S_{2n}(q)\frac{x^{2n}}{(q;q)_{2n}}=\frac{1}{\cos_q(x)}. 
 $$
 The following is our second $q$-analog of~\eqref{eq:Tan2}.
\begin{theorem}\label{thm:q-analog2}
For $n\geq0$, we have 
\begin{equation}\label{q2:Tan2}
\sum_{k=0}^{n}(-1)^{k}{2n+1\brack 2k}(-q;q)_{2n-2k} T_{2k+1}(q)=(-1)^n\widehat T_{2n+1}(q),
\end{equation}
where 
\begin{equation}
(-1)^n\widehat T_{2n+1}(q)=\sum_{k=0}^{n} (-1)^{k} {2n+1\brack 2k}q^{2k} S_{2k}(q).
\end{equation}
Equivalently, 
$$
\sum_{k=0}^{n}(-1)^{k}{2n+1\brack 2k}\bigl((-q;q)_{2n-2k} T_{2k+1}(q)-q^{2k} S_{2k}(q)\bigr)=0.
$$
\end{theorem}

To see that identity~\eqref{q2:Tan2} is a $q$-analog of ~\eqref{eq:Tan2}, we prove the following identity connecting $q$-tangent numbers with $q$-secant numbers, which seems first appears in the work of Huber and Yee~\cite{Yee}. 

\begin{theorem}\label{thm:q-sec:tan}
For $n\geq0$, we have
\begin{equation}\label{eq:q-sec:tan}
\sum_{k=0}^n(-1)^k{2n+1\brack 2k}S_{2k}(q)=(-1)^nT_{2n+1}(q). 
\end{equation}
\end{theorem}

The following combinatorial object is crucial in our proofs of the above two theorems. 
\begin{definition}
A \blue{\bf\em permutation pair on $[n]$} is a pair $(\pi,\sigma)$ of words (possibly empty) such that the word concatenation $\pi\cdot\sigma$ forms a permutation of $[n]$. For instance, $(2571,86493)$ and $(\emptyset, 257186493)$ are two different permutation pairs on $[9]$. 
For a permutation pair $(\pi,\sigma)$, introduce the \blue{\bf\em sign} and the \blue{\bf\em inversion number} of $(\pi,\sigma)$ as
$$
\sgn(\pi,\sigma):=(-1)^{\lfloor k/2\rfloor}\quad\text{and}\quad\inv(\pi,\sigma):=\inv(\pi\cdot\sigma),
$$
where $k$ is the size of $\sigma$. 
\end{definition}

First we prove Theorem~\ref{thm:q-sec:tan}. 

\begin{proof}[{\bf Proof of Theorem~\ref{thm:q-sec:tan}}]
Let us consider the set $\IP_{2n+1}$ of all permutation pairs $(\pi,\sigma) $ on $[2n+1]$ such that $\pi$ is increasing and $\sigma$ is up-down of even size. 
By the interpretation~\eqref{eq:qmul} of $q$-binomial coefficients, we have 
\begin{equation}\label{IP:left}
\sum_{(\pi,\sigma)\in\IP_{2n+1}}\sgn(\pi,\sigma)q^{\inv(\pi,\sigma)}=\sum_{k=0}^n(-1)^k{2n+1\brack 2k}S_{2k}(q). 
\end{equation}

On the other hand, for $(\pi,\sigma)\in\IP_{2n+1}$ with $\pi=\pi_1\pi_2\cdots\pi_{2n+1-2k}$ and $\sigma=\sigma_1\sigma_2\cdots\sigma_{2k}$ for some $k\geq0$, we introduce the mapping 
$$
f(\pi,\sigma)=
\begin{cases}
(\pi\sigma_1\sigma_2,\sigma_3\cdots\sigma_{2k}),
\quad&\text{if $\pi_{2n+1-2k}<\sigma_1$;}\\
(\pi_1\cdots\pi_{2n-1-2k},\pi_{2n-2k}\pi_{2n+1-2k}\sigma),\qquad&\text{if $\pi_{2n+1-2k}>\sigma_1$ and $k<n$;}\\
(\pi,\sigma),\quad&\text{ if $\pi_{2n+1-2k}>\sigma_1$ and $k=n$}.
\end{cases}
$$
Let $\IP^*_{2n+1}$ be the set of permutation pairs in $\IP_{2n+1}$ that satisfy the condition in the third case above. Then, the mapping $f$ is an involution on $\IP_{2n+1}$ preserving inversion numbers, fixing elements in $\IP^*_{2n+1}$ and sign-reversing on $\IP_{2n+1}\setminus\IP^*_{2n+1}$. Thus,
$$
\sum_{(\pi,\sigma)\in\IP_{2n+1}}\sgn(\pi,\sigma)q^{\inv(\pi,\sigma)}=\sum_{(\pi,\sigma)\in\IP^*_{2n+1}}\sgn(\pi,\sigma)q^{\inv(\pi,\sigma)}=(-1)^nT_{2n+1}(q),
$$
which completes the proof in view of~\eqref{IP:left}. 
\end{proof}

In similar spirit, we can also prove Theorem~\ref{thm:q-analog2}. 

\begin{proof}[{\bf Proof of Theorem~\ref{thm:q-analog2}}]
Let us consider the set $\UP_{2n+2}$ of all permutation pairs $(\pi,\sigma) $ on $[2n+2]$ such that
\begin{itemize}
\item $\pi$ is unimodal;
\item $\sigma$ is down-up of odd size and contains the letter $2n+2$. 
 \end{itemize}
By the interpretation~\eqref{eq:qmul} of $q$-binomial coefficients and Lemma~\ref{lem:unimodal}, we have 
\begin{equation}\label{UP:left}
\sum_{(\pi,\sigma)\in\UP_{2n+2}}\sgn(\pi,\sigma)q^{\inv(\pi,\sigma)}=\sum_{k=0}^n(-1)^k{2n+1\brack 2k}(-q;q)_{2n-2k}T_{2k+1}(q). 
\end{equation}

On the other hand, for $(\pi,\sigma)\in\UP_{2n+2}$ with $\pi=\pi_1\pi_2\cdots\pi_{l}$ and $\sigma=\sigma_1\sigma_2\cdots\sigma_{k}$ for some odd positive integer $k$ satisfying $k+l=2n+2$, we introduce the mapping 
$$
g(\pi,\sigma)=
\begin{cases}
(\pi\sigma_1\sigma_2,\sigma_3\cdots\sigma_{k}),
\quad&\text{if $\pi_{l}>\sigma_1$ or $\pi_{l}<\sigma_1\neq 2n+2$ with $\pi$ increasing;}\\
(\pi_1\cdots\pi_{l-2},\pi_{l-1}\pi_{l}\sigma)\quad&\text{if $\pi_l<\sigma_1$ and $\pi_{l-1}>\pi_l$;}\\
(\pi,\sigma),\quad&\text{if $\sigma_1=2n+2$ and $\pi$ is increasing}.
\end{cases}
$$
Let $\UP^*_{2n+2}$ be the set of all permutation pairs in $\UP_{2n+2}$ that satisfy the condition in the third case above. Then, the mapping $g$ is an involution on $\UP_{2n+2}$ preserving inversion numbers, fixing elements in $\UP^*_{2n+2}$ and sign-reversing on $\UP_{2n+2}\setminus\UP^*_{2n+2}$. Thus,
\begin{align*}
\sum_{(\pi,\sigma)\in\UP_{2n+2}}\sgn(\pi,\sigma)q^{\inv(\pi,\sigma)}&=\sum_{(\pi,\sigma)\in\UP^*_{2n+2}}\sgn(\pi,\sigma)q^{\inv(\pi,\sigma)}\\
&=\sum_{k=0}^n(-1)^k{2n+1\brack 2k}q^{2k}S_{2k}(q).
\end{align*}
This completes the proof in view of~\eqref{UP:left}. 
\end{proof}

Finally, we emphasize that our approach via permutation pairs also works for Huber and Yee's identities~\cite{Yee} for some different $q$-secant and $q$-tangent numbers involving inversions and half descents. For simplicity, we only illustrate the ideas for one of their identity concerning the $q$-secant numbers 
$$
S^o_{2n}(q):=\sum_{\pi\in\A'_{2n}}q^{\inv(\pi)+\des(\pi_o)}\quad\text{for $n\geq1$},
$$
where $\des(\pi_o):=|\{i\in[n-1]:\pi_{2i-1}>\pi_{2i+1}\}|$, called {\em half descents} of $\pi$ with odd indices. 
\begin{theorem}[\text{Huber and Yee~\cite[Theorem~3.1]{Yee}}]
Let $S^o_{0}(q):=q^{-1}$. Then
\begin{equation}\label{eq:yee}
\sum_{k=0}^n{2n\brack 2k}(-1)^kq^{(n-k)^2}S^o_{2k}(q)=0.
\end{equation}
\end{theorem}
\begin{proof}
Let us consider the set $\CP_{2n}$ of all permutation pairs $(\pi,\sigma) $ on $[2n]$ such that
\begin{itemize}
\item $\pi=\pi_1\pi_2\cdots\pi_{2k}$ is up-down with $k\geq0$;
\item and $\sigma=\sigma_1\sigma_2\cdots\sigma_{2l}$ is {\em compressed up-down} with $k+l=n$, that is, 
$$
\sigma_1<\sigma_3<\cdots<\sigma_{2l-1}<\sigma_{2l}<\sigma_{2l-2}<\cdots<\sigma_{4}<\sigma_2. 
$$ 
 \end{itemize}
Note that $\inv(\sigma)=l^2-l$. Define the weight of $(\pi,\sigma)$ as 
$$
\wt(\pi,\sigma):=\inv(\pi,\sigma)+\des(\pi_o)-\chi(k=0)+l.
$$
By the interpretation~\eqref{eq:qmul} of $q$-binomial coefficients, we have 
\begin{equation}\label{CP:left}
\sum_{(\pi,\sigma)\in\CP_{2n}}\sgn(\pi,\sigma)q^{\wt(\pi,\sigma)}=\sum_{k=0}^n(-1)^kq^{(n-k)^2}{2n\brack 2k}S^o_{2k}(q). 
\end{equation}
It remains to construct a weight preserving sign-reversing involution on $\CP_{2n}$. 
Given $(\pi,\sigma)\in\CP_{2n}$ with $\pi=\pi_1\pi_2\cdots\pi_{2k}$ and $\sigma=\sigma_1\sigma_2\cdots\sigma_{2l}$, we construct the mapping $h: (\pi,\sigma)\mapsto h(\pi,\sigma)$ according to the following four cases.
\begin{itemize}
\item[(a-1)] \blue{$\pi_{2k-1}>\sigma_1$ (the case $k=0$ is included in this case, for convenience's sake).} In this case, set $h(\pi,\sigma)=(\pi\sigma_1\sigma_2,\sigma_3\cdots\sigma_{2l})$. 
\item[(a-2)] \blue{$\pi_{2k-1}<\sigma_1$, $\pi_{2k}>\sigma_2$ and either $\pi_{2k-3}>\pi_{2k-1}$ or $k=1$ (the case $l=0$ and either $\pi_{2k-3}>\pi_{2k-1}$ or $k=1$ is included in this case).} In this case, set $h(\pi,\sigma)=(\pi_1\pi_2\cdots\pi_{2k-2},\pi_{2k-1}\pi_{2k}\sigma)$. 
\item[(b-1)] \blue{$\pi_{2k-1}<\sigma_1$, $\pi_{2k}>\sigma_2$ and $\pi_{2k-3}<\pi_{2k-1}$ (the case $l=0$ and $\pi_{2k-3}<\pi_{2k-1}$ is included in this case).} In this case, let $(\pi',\sigma')$ be obtained from $(\pi,\sigma)$ by switching the letter $\pi_{2k-2}$ with the greatest letter in $\{\pi_{2k-1},\sigma_1,\sigma_2,\ldots,\sigma_{2l}\}$ that is smaller than $\pi_{2k-2}$. Thus, $\inv(\pi',\sigma')=\inv(\pi,\sigma)-1$. Now set $h(\pi,\sigma)=(\pi_1'\pi_2'\cdots\pi_{2k-2}',\pi_{2k-1}'\pi_{2k}'\sigma')$. 
\item[(b-2)] \blue{$\pi_{2k-1}<\sigma_1$ and $\pi_{2k}<\sigma_2$.} In this case, let $(\pi',\sigma')$ be obtained from $(\pi,\sigma)$ by switching the letter $\pi_{2k}$ with the smallest letter in $\sigma$ that is greater than $\pi_{2k}$. Thus, $\inv(\pi',\sigma')=\inv(\pi,\sigma)+1$. Now set 
$h(\pi,\sigma)=(\pi'\sigma_1'\sigma_2',\sigma_3'\sigma_4'\cdots\sigma'_{2l})$. 
\end{itemize}
It is routine to check that $h$ is a weight preserving sign-reversing involution on $\CP_{2n}$. This completes the proof of the theorem. 
\end{proof}

\section*{Acknowledgement} 
The work presented here was initiated while the second author was visiting the first author at KAIST and the Institute for Basic Science (IBS) in Daejeon, South Korea. The authors thank Prof. Sang-il Oum for supporting the first author's research visit to IBS. This work was also supported by the National Science Foundation of China grant 12322115.

\end{document}